\documentclass[12pt]{amsart} 
\usepackage{amsmath,amssymb,latexsym,amsthm,tikz, enumerate,amscd,hyperref,mathtools} 
\usepackage[abbrev]{amsrefs} 
\usepackage{graphicx}
\usepackage{pinlabel}
\usepackage{tikz-cd}
\usepackage[margin=1in]{geometry}

\newtheoremstyle{dotless}{}{}{\itshape}{}{\bfseries}{}{ }{}

\newtheorem{Theorem}{Theorem}[section]
\newtheorem{Lemma}[Theorem]{Lemma}

\theoremstyle{definition} 
\newtheorem{definition}[Theorem]{Definition} 

\newtheorem*{acknowledgements}{Acknowledgements}

\newtheorem*{Remark}{Remark}

\newtheorem*{Example}{Example}

\DeclareMathOperator{\piprod}{\raisebox{-0.1em}{\huge{$\pi$}}\kern -0.2em}

\newcommand{\cu}{\mathcal {U}}

\newcommand{\nn}{{\mathbb N}}

\newcommand{\rr}{{\mathbb R}}

\newcommand{\zz}{{\mathbb Z}}

\newcommand{\Fp}{{\mathbb F_p}}

\newcommand{\gdim}{\operatorname{gdim}}

\newcommand{\Cone}{\operatorname{Cone}}

\def\clap#1{\hbox to 0pt{\hss#1\hss}}

\newcommand{\comment}[1]{}

\newcommand{\gs}{\sigma} 
 
\newcommand{\gt}{\tau}

\newcommand{\gG}{\Gamma}

\newcommand{\id}{\operatorname{id}}

\newcommand{\ent}{\operatorname{ent}}

\newcommand{\vol}{\operatorname{Vol}}

	{ 
\end{enumerate}
} 

	{ 
\end{enumerate}
} 

\newenvironment{enumeratei'}{ 
\begin{enumerate}[\upshape (i)$'$]}
	{ 
\end{enumerate}
} 

\newenvironment{enumerate1'}{ 
\begin{enumerate}[\upshape (1)$'$]}
	{ 
\end{enumerate}
} 

\newenvironment{enumeratea'}{ 
\begin{enumerate}[\upshape (a)$'$]}{ 
\end{enumerate}
}

\usepackage{color}
\usepackage[outerbars,color]{changebar}
\usepackage{ifthen}

\newboolean{usecolor}
\setboolean{usecolor}{true}

\ifthenelse{\boolean{usecolor}}
{
  \definecolor{colore}{cmyk}{0,1,0.6,0}
  \definecolor{coloregen}{cmyk}{0.7,0,1,0}
  \definecolor{coloresimo}{cmyk}{1,0.6,0,0}
}
{
  \definecolor{colore}{cmyk}{0,0,0,1}
  \definecolor{coloregen}{cmyk}{0,0,0,1}
  \definecolor{coloresimo}{cmyk}{0,0,0,1}
}

\numberwithin{equation}{section} 
\usepackage{ifthen}

\newcommand{\showcomments}{yes}

\newsavebox{\commentbox}
{\ifthenelse{\equal{\showcomments}{yes}}%
{\footnotemark
    \begin{lrbox}{\commentbox}
    \begin{minipage}[t]{1.25in}\raggedright\sffamily\upshape\tiny
    \footnotemark[\arabic{footnote}]}
{\begin{lrbox}{\commentbox}}}%
{\ifthenelse{\equal{\showcomments}{yes}}%
{\end{minipage}\end{lrbox}\marginpar{\usebox{\commentbox}}}
{\end{lrbox}}}

\title{Minimal volume entropy of RAAG's}
\author{Matthew Haulmark}
\author{Kevin Schreve}

\begin{document}
\maketitle
\begin{abstract}
Bregman and Clay recently characterized which right-angled Artin groups with geometric dimension $2$ have vanishing minimal volume entropy. In this note, we extend this characterization to higher dimensions. \end{abstract}
\section{Introduction}

Let $X$ be a finite complex with a piecewise Riemannian metric $g$ (i.e. a collection of Riemannian metrics on cells which agree on intersections). Fix a basepoint $x_0$ in the universal cover $\widetilde X$, and let $\tilde g$ be the pulled-back metric on $\widetilde X$. The associated \emph{volume entropy} of $(X,g)$ is the exponential growth rate of the balls $B_{x_0}(t)$ in the universal cover: $$\ent(X,g) = \lim_{t \rightarrow \infty} \frac{1}{t} \log \vol(B_{x_0}(t), \tilde g)$$ 

This limit does not depend on the choice of basepoint.
We now define the \emph{minimal volume entropy} $\omega(X)$ to be $$\omega(X) = \inf_g \ent(X,g) \vol(X,g)^{1/\dim X}$$ 
where we minimize over all piecewise Riemannian metrics $g$. Normalizing by the volume guarantees this does not change under scaling $g$. This invariant was initially defined for Riemannian manifolds in \cite{gromov}. 

Now, suppose $G$ is a group with a finite classifying space $BG$. Let $\gdim(G)$ be the \emph{geometric dimension} of $G$, i.e. $\gdim(G)$ is the minimal dimension of such a $BG$.  We define the minimal volume entropy of $G$, denoted $\omega(G)$, to be the infimum of $\omega(BG)$ over all finite classifying spaces $BG$ of dimension = $\gdim(G)$. We say a classifying space $BG$ is  \emph{minimal dimensional} if $\dim BG = \gdim(G)$.

In this note, we study this invariant for right-angled Artin groups (from now on RAAG's). 
If $L$ is a flag simplicial complex, recall that the associated RAAG $A_L$ has a presentation with generators corresponding to vertices, and where two generators commute iff the two vertices span an edge in $L$. We give an almost complete characterization of the (non)vanishing of $\omega(A_L)$ based on the topology of the defining flag complex $L$. The geometric dimension of $A_L$ is equal to $\dim L + 1$ (an $n$-simplex in $L$ corresponds to an $\zz^{n+1}$ subgroup of $A_L$, so this is an obvious lower bound). We prove the following theorem:

\begin{Theorem}
\label{t:main}
Let $L$ be a $d$-dimensional flag complex and $A_L$ be the corresponding RAAG. Then 

\begin{enumerate}
\item If $H^d(L, \zz) \ne 0$, then $\omega(A_L) > 0$. 
\item If $L$ embeds into a $d$-dimensional contractible complex, then $\omega(A_L) = 0$. 
\end{enumerate}
\end{Theorem}

Bregman and Clay had previously proved this theorem when $L$ is a flag graph \cite{bc21}. 
If $d \ne 2$, then $H^d(L,\zz) = 0$ is equivalent to $L$ embedding into a $d$-dimensional contractible complex, hence these conditions are complementary [Remark 1.26, \cite{dgo}]. This embedding condition was used in \cite{dgo} to construct ``low" dimensional manifold models for $BA_L$. By the universal coefficient theorem, $H^d(L,\zz) = 0$ is equivalent to $H_d(L, \zz) = 0$ and $H_{d-1}(L, \zz)$ being free abelian. Also, by universal coefficients again, it's equivalent to $H_d(L, \mathbb F_p) = 0$ for all primes $p$.

\begin{Example}
Here is a curious example. Suppose $L_1$ is a flag triangulation of $\rr P^2$, and suppose that $L_2$ is a flag triangulation of a $2$-dimensional $\zz/3$-Moore space (for instance $L_2$ is obtained by attaching a disc to a circle by a degree $3$-map). By Theorem \ref{t:main}, both $A_{L_1}$ and $A_{L_2}$ have nonvanishing minimal volume entropy. On the other hand, again by Theorem \ref{t:main}, their product has vanishing minimal volume entropy (as the join $L_1 \ast L_2$ has $H_5(L, \Fp) = 0$ for all $p$). This is in contrast with the simplicial volume $||M||$ of a closed manifold $M$; Gromov proved in \cite{gromov} the inequality $$||M||||N|| \le || M \times N|| \le {\dim M + \dim N \choose \dim M} ||M||||N||.$$

For a closed $m$-manifold $M$, Gromov also showed the inequality $$\omega(M)^m \ge C_m ||M||$$ for a constant $C_m$ only dependent on $m$. In particular, the product of any two manifolds $M \times N$ with $\omega(M), \omega(N) > 0$  has $\omega (M \times N) > 0$. 
\end{Example}

To prove Theorem \ref{t:main}, we use the following fibering criteria of Babenko-Sabourau \cite{bs21}. This was previously used by Bregman and Clay \cite{bc21} to compute $\omega$ of RAAG's based on flag graphs. 

\begin{definition}
Let $X$ be a simplicial complex. We say $X$ has FCA (fiber collapsing assumption) if there is a simplicial complex $P$ with $\dim P < \dim X$ and a simplicial map $f : X \rightarrow P$ so that for all $p \in P$, if $F_p$ is a component of $f^{-1}(p)$ then the image subgroup $i_\ast(\pi_1(F_p))$ in $\pi_1(X)$ is subexponentially growing with subexponential growth rate $< 1- \frac{\dim P}{\dim X}$. 
\end{definition}

Babenko and Sabourau \cite{bs21} show that if $X$ has FCA then $\omega(X) = 0$. If $X$ is a minimal dimensional classifying space for $G$, this of course implies that $\omega(G) = 0$. 
In our setting, all subexponentially growing subgroups of RAAG's are free abelian (this follows from a theorem of Baudisch described below). The subexponential growth rate of a finitely generated free abelian group is 0, so we do not require going into the details of this rate. 

\begin{definition}
A group $G$ has \emph{uniform exponential growth} if there is a $\delta > 0$ so that the growth of $G$ with respect to any finite generating set $S$ is exponential with growth rate $> \delta$. 
We say $X$ has FNCA (fiber non-collapsing assumption) if there is a $\delta > 0$ so that every map $f:X \rightarrow P$ with $\dim P < \dim X$ has a point preimage with the induced subgroup $\pi_1^{-1}(f(p))$ in $\pi_1(X)$ having uniform exponential growth $> \delta$. 
\end{definition}

Babenko and Sabourau \cite{bs21} also show that if $X$ has FNCA, then $\omega(X) > \epsilon_m$, where $\epsilon_m$ only depends on the dimension $m$ of $X$ and $\delta$. 
The conditions FCA and FNCA are almost, but not quite, complementary. 
On the other hand, Bregman and Clay \cite{bc21} show that they are for classifying spaces of RAAG's. This follows from the following facts: RAAG's have so-called \emph{uniform uniform exponential growth}, in the sense that there is a $\delta > 0$ so that every non-abelian finitely generated subgroup of a given RAAG  has uniform exponential growth $> \delta$ (every nonabelian subgroup on two generators is free by a theorem of Baudisch \cite{baudisch}). Now, if a classifying space for a RAAG does not have the FCA, then any map to a smaller dimensional complex has a point preimage which is nonabelian. Hence, this subgroup has uniform exponential growth $> \delta$, and in particular has the FNCA. It follows that for a RAAG $A_L$, if we can show that every minimal dimensional classifying space $BA_L$ has FNCA, this will imply that $\omega(A_L) > 0$ (in fact, in this case it suffices to exhibit FNCA for one model of $BA_L$, see Proposition 3.9 of \cite{bc21}).

Part (2) of Theorem \ref{t:main} follows from an explicit construction of $BA_L$. 
This model for $BA_L$ is built by gluing together tori of various dimensions corresponding to the simplices of $L$. 
If $L$ embeds into a contractible complex $L'$ of the same dimension, then this model for $BA_L$ naturally maps to $L'
$, and the preimages of points are homotopic to tori or points. 
Therefore, this model for $BA_L$ has FCA. To prove part (1), we consider the mod $p$ homology growth in residual chains of finite index subgroups. For RAAG's, this was recently computed by Avramidi, Okun and the second author \cite{aos}. If there is a minimal dimensional model for $BA_L$ with FCA, it will follow from a result of Sauer's that this growth vanishes in the top dimension for all $p$ \cite{s16}. The computation in \cite{aos} then shows that the top $\Fp$-homology of $L$ vanishes for all $p$.

\begin{acknowledgements} We thank Roman Sauer for answering a question about his work in \cite{s06} and \cite{s16}, and Matt Clay for helpful comments on an earlier draft. 

\end{acknowledgements}

\section{Classifying spaces of RAAG's}\label{s:bal}

Our calculation of $\omega(A_L)$ relies on the following construction of models for $BA_L$. Let $L$ be the defining flag complex of the RAAG, and let $K$ be the geometric realization of the poset of simplices of $L$. 
Then $K$ is isomorphic to the cone on the barycentric subdivision of $L$, where the cone point corresponds to the empty simplex. 
Given a simplex $\gs \in L$, let $A_\gs$ be the corresponding free abelian subgroup of $A_L$. 
A point $x \in K$ is contained in some minimal simplex, which corresponds to a chain of simplices of $L$. 
Let $\gs(x)$ be the smallest element in this chain. The \emph{basic construction} $\cu(A_L, K)$ is defined to be $$\cu(A_L, K) := A_L \times K/\sim$$ where $(g,x) \sim (g',x')$ if and only if $x = x'$ and $gA_{\gs(x)} = g'A_{\gs(x)}$. 

Then $G$ acts on $\cu(A_L, K)$ with strict fundamental domain $1 \times K$, which we identify with $K$\footnote{A strict fundamental domain for a group action on a CW-complex is a subcomplex which intersects each orbit in a single point.}. For RAAG's, it is known that $\cu(A_L, K)$ is contractible, in fact it admits a natural CAT(0) cubical structure. 

The stabilizer of a simplex $\tau \subset K$ is the free abelian special subgroup $A_{\min \gt}$, where $\min \gt$ is the smallest element in the corresponding chain of simplices of $L$. In particular, the stabilizer is trivial if and only if $\gt$ contains the cone point corresponding to the empty simplex. 

Let $\cu=\cu(A_L,K)$. To build a classifying space $BA_L$, we use the Borel Construction $EA_L \times_{A_L} \cu$ (if $X$ and $Y$ are two $G$-spaces, then $X \times_G Y$ is the direct product $X \times Y$ quotiented by the diagonal $G$-action). 
Since $\cu$ is contractible, this produces a (noncompact) model of $BA_L$. 

The action of the stabilizer of any cell in $\cu$ fixes the cell. 
It follows that $EA_L \times_{A_L} \cu$ naturally maps to $\cu/A_L = K$.
The preimage of an open simplex $\gt$ in $K$ is homeomorphic to $\tau \times EA_L/A_{\min \gt}$, hence homotopy equivalent to a torus $T^{\min \tau}$. A rebuilding procedure of Geoghegan lets us build a compact model for $BA_L$, denoted $X_L$, which maps to $K$ and the preimage of a simplex $\gt$ is homeomorphic to $\gt \times T^{\min \gt}$ [Chapter 6, \cite{g08}]. Note that $X_L$ is $(\dim L + 1)$-dimensional; it follows that the geometric dimension of $A_L$ is $\dim L + 1$. 

In particular, we can assume that if $\gt$ is a simplex in  $K$ which contains the cone point, then the preimage in $X_L$ is a copy of $\gt$. Therefore, $X_L = \Cone(L) \cup Y_L$, where $Y_L$ is the preimage of $L$ in $BA_L$. We say $Y_L$ is the \emph{toral subcomplex} of $X_L$. The following lemma is immediate from the discussion. 

\begin{Lemma}\label{l:contract}
If $L$ is contractible, then $Y_L$ is a model for $BA_L$. 
\end{Lemma}

\section{$\omega(A_L)$}

\subsection{Vanishing of entropy for RAAG's}\label{s:vanishing}

\begin{Theorem}\label{t:vanishing}
Let $A_L$ be a RAAG based on a $d$-dimensional flag complex $L$. Suppose that $L$ embeds into a $d$-dimensional contractible complex. Then $\omega(A_L) = 0$.
\end{Theorem}
\begin{proof}
First, suppose that $L$ is contractible. We use the model $X_L$ of $BA_L$ as in the previous section. In particular, Lemma \ref{l:contract} guarantees a $(d+1)$-dimensional model for $BA_L$ which projects to $L$, and the preimage of a point in $L$ is homotopy equivalent to a torus of some dimension. Therefore, this model satisfies $FCA$, and hence $\omega(A_L) = 0$ \cite{bs21}. 

Now, suppose $L$ is $d$-dimensional and embeds into a $d$-dimensional contractible complex $L'$. Let $Y_L$ be the toral subcomplex for $L$. Then a model for $BA_L$  can be obtained by forming the amalgam $X_L = Y_L \cup_L L'$ (this follows from Lemma 1.21 and Lemma 1.23 in \cite{dgo}, see the discussion there in Section 1.4). This is $(d+1)$-dimensional, and naturally projects to $L'$. The preimage of a point is homotopy equivalent to a torus (a point if the point is in $L' - L$). This shows that this model satisfies $FCA$, and hence $\omega(A_L) = 0$. 
\end{proof}

\subsection{Non-vanishing of entropy for RAAG's}

We now show that non-vanishing of $\omega$ for RAAG's follows from work of Sauer on mod $p$-homology growth in residual, finite index normal chains \cite{s16}, and the calculation of this growth for RAAG's in \cite{aos}. 

First, assume that $X$ is a simplicial complex with residually finite fundamental group $G$. Let $$G = \gG_0 \ge \gG_1 \ge \gG_2 \ge \dots$$ be a chain of finite index normal subgroups with $\cap_k \gG_k = 1$, and let $X_k$ be the corresponding covers of $X$ (we say $\gG_k$ is a \emph{residual chain}). We say that $X$ has nonvanishing $\Fp$-homology growth in degree $i$ if $$\limsup_{k} \frac{b_i(X_k, \Fp)}{[G:\gG_k]} > 0$$

If one instead considers the $\mathbb Q$-homology growth, the limsup coincides with the $i^{th}$ $L^2$-Betti number of $X$ by a theorem of L{\"u}ck \cite{lueck}. If $X = BG$ is aspherical, then we will write $b_i(\gG_k,\Fp)$ instead of $b_i(B\gG_k,\Fp)$. 

Now, suppose that we have a simplicial complex $X$ which satisfies FCA; hence we have a map $f: X \rightarrow P$ with $\dim P < \dim X$. If we cover $P$ by open stars of vertices, then we can use $f$ to pull back this cover to $X$. In particular, if we cover $X$ by connected components of preimages of open stars, the multiplicity of this cover is equal to $\dim P + 1 \le \dim X$. Each open set in the cover deformation retracts to a connected component of $f^{-1}(p)$.  Therefore, the image of the fundamental group of each set in the cover is a subexponentially growing subgroup of $\pi_1(X)$.  More generally, we say an open set $U$ in a topological space $X$ is \emph{amenable} if the image subgroup $i_\ast \pi_1(U)$ in $\pi_1(X)$ is amenable.

The following theorem follows immediately from the proof of Theorem 1.8 in \cite{s16}

\begin{Theorem}
Let $G$ be residually finite, and let $\gG_k$ be a residual sequence of finite index normal subgroups. Suppose that there is a finite $BG$ with an amenable cover of multiplicity $n$. Then $$\lim_{k \rightarrow \infty} \frac{b_i(\gG_k, \Fp)}{[G:\gG_k]} = 0 \text{ for } i \ge n.$$

\end{Theorem}

In [Theorem 1.8, \cite{s16}], Sauer shows that if an aspherical $n$-manifold $M$ has an open cover by amenable sets of multiplicity $n$, then the $\Fp$-homology growth of $\pi_1(M^n)$ vanishes in all degrees (for all finite index normal chains). To do this, Sauer constructs complexes $S(k)$ which homotopy retract\footnote{i.e. there are maps $f: \widetilde M^n/\gG_k \rightarrow S(k)$ and $g: S(k) \rightarrow \widetilde M^n/\gG_k$ so that $g \circ f \sim \id_{\widetilde M^n/\gG_k}$} onto $\widetilde M^n/\gG_k$ and have sublinear (in $[G:\gG_k]$) number of $n$-cells. The construction of these complexes does not require the manifold structure.
Therefore, the same argument shows that if an aspherical $n$-complex $X$ has an open cover by amenable sets of multiplicity $n$, then the $\Fp$-homology growth of $\pi_1(X)$ vanishes in degree $n$ (for the usual $L^2$-Betti numbers, this is Theorem C of \cite{s06}). 
).

In \cite{aos}, Avramidi, Okun, and the second author computed the $\Fp$-homology growth of RAAG's; the analogous theorem for $b_i^{(2)}(A_L)$ was proved earlier by Davis and Leary \cite{dl}.

\begin{Theorem}\label{t:aos}
Let $A_L$ be a RAAG based on a flag complex $L$. Let $\{\gG_k\}_{k \in \nn}$ be a residual chain of finite index normal subgroups. 
Then $$\lim_{k \rightarrow \infty} \frac{b_i(\gG_k, \Fp)}{[A_L:\gG_k]} = \bar b_{i-1}(L, \Fp),$$ where $\bar b_{i-1}(L, \Fp)$ is the reduced Betti number of $L$ with $\Fp$-coefficients.
\end{Theorem}

By Theorem \ref{t:aos}, if $H_d(L, \Fp) \ne 0$, then $A_L$ has nonvanishing $\Fp$-homology growth in dimension $d+1$. By the above,  any minimal dimensional classifying space will not have FCA. Therefore, any minimal dimensional classifying space has FNCA, and by Babenko and Sabourau's results (plus the fact that RAAG's have uniform uniform exponential growth \cite{baudisch} we obtain $\omega(A_L) > 0$. 
This gives the following characterization of non-vanishing volume growth entropy of RAAG's.

\begin{Theorem}\label{t:nonvanishing}
Suppose that $A_L$ is a RAAG based on a $d$-dimensional flag complex $L$. 
If $H_{d+1}(L, \mathbb F_p) \ne 0$ for any $p$, then $\omega(A_L) > 0$. Hence if $H^d(L, \zz) \ne 0$, then $\omega(A_L) > 0$. 
\end{Theorem}

\begin{Remark}
Matt Clay pointed out to us that Proposition 3.9 in \cite{bc21} implies that $\omega(A_L) > \epsilon$ for $\epsilon$ only depending on $\dim L$. 
\end{Remark}

\begin{Remark}
The computation of $\Fp$-homology growth in \cite{aos} was extended in \cite{os}. The correct context for these computations seems to be a group $G$ acting on a contractible complex with strict fundamental domain $Q$. For example, a similar computation works for any residually finite Artin group which satisfies the $K(\pi,1)$-conjecture. By the same argument as above, if the nerve of these Artin groups has $H^{\dim L}(L, \zz) \ne 0$, then the minimal volume entropy will be strictly positive. The corresponding vanishing result for Artin groups seems harder to prove. 

\end{Remark}


\begin{bibdiv}
	\begin{biblist}
		
		
		
		
		\bib{aos}{article}{ 
		author={Avramidi, Grigori}, 
		author={Okun, Boris}, 
		author={Schreve, Kevin}, 
		title={Mod $p$ and torsion homology growth in nonpositive curvature}, 
		journal = {Inventiones Mathematicae}, 
		volume = {226}, 
		number = {15}
		date={2021}, 
		}
		
		\bib{bs21}{article}{ 
		author={Babenko, Ivan}, 
		author={Sabourau, Stephane}, 
		title={Minimal volume entropy of simplicial complexes}, 
		note = {preprint, arXiv:math/2002.11069}, 
		
		}
		
			\bib{baudisch}{article}{ 
		author={Baudisch, Andreas}, 

		title={Subgroups of semi-free groups}, 
		journal = {Acta Math. Acad. Sci. Hungar.}, 
		date={1981}, 
		volume={38},
		pages = {19-28},
		}
		
		\bib{bc21}{article}{ 
		author={Bregman, Corey}, 
		author={Clay, Matt}, 
		title={Minimal volume entropy of free-by-cyclic groups and 2-dimensional right-angled Artin groups}, 
		journal = {Mathematische Annalen}, 
		date={2021}, 
		}
		
		
		 \bib{dl}{article} {

	author = {Davis, M. W.},
	author = {Leary, I. J.},
	title = {The {$l^2$}-cohomology of {A}rtin groups},
	journal = {J. London Math. Soc. (2)},

	volume = {68}, YEAR = {2003},
	number = {2},
	pages = {493--510},
	
}

 \bib{dgo}{article} {

	author = {Davis, M. W.},
	author = {Le, Giang},
	author = {Schreve, Kevin},
	title = {The action dimension of simple complexes of groups},
	journal = {Journal of Topology},

	volume = {12}, YEAR = {2019},

	pages = {1266-1314},
	
}

 \bib{g08}{book}{
	author = {Geoghegan, Ross},
	title = {Topological methods in group theory},
	series = {Graduate Texts in Mathematics}, publisher={Springer, New York},
	date = {2008},
	volume = {243}, ISBN={978-0-387-74611-1},
	url = {http://dx.doi.org/10.1007/978-0-387-74614-2},
 }
 
  \bib{gromov}{article}{
	author = {Gromov, Misha},
	title = {Volume and Bounded Cohomology},
	journal = {Inst. Hautes Études Sci. Publ. Math},
  number = {56}, 
YEAR = {1983},

	pages = {5-99},
 }
 
 




 \bib{lueck}{article}{
	author = {L\"uck, Wolfgang},
	title = {Approximating ${L}\sp 2$-invariants by their finite-dimensional analogues},
	date = {1994},
	issn = {1016-443X},
	journal = {Geom. Funct. Anal.} ,
	volume = {4},
	number = {4},
	pages = {455\ndash 481},
 }
 

		
	
		
		\bib{os}{article}{
		author={Okun, Boris}, 
		author={Schreve, Kevin}, 
		title={Torsion invariants of complexes of groups}, 
		journal = {arXiv:2108.08892},
		date={2021}, 
		}
		
		
		\bib{s06}{article}{ 
		author={Sauer, Roman},  
		title={Amenable covers, volume and L2-Betti numbers of aspherical manifolds}, 
		journal = {J. Reine Angew. Math},
		date={2009}, 
		volume = {636},
		pages = {47-92},
		}
		
		\bib{s16}{article}{ 
		author={Sauer, Roman},  
		title={Volume and homology growth of aspherical manifolds}, 
		journal = {Geom. Topol.},
		date={2016}, 
		volume = {20},
		pages = {1035-1059},
		}
		
		

\end{biblist}
\end{bibdiv}
\end{document}